\documentclass[12pt]{article}
\usepackage{geometry}                
\usepackage[parfill]{parskip}    
\usepackage{graphicx}
\usepackage{amssymb}
\usepackage{epstopdf}
\usepackage{amsmath}
\usepackage{amsthm}
\usepackage{dsfont}
\newtheorem{thm}{Theorem}[section]
\newtheorem{lem}[thm]{Lemma}
\newtheorem{prop}[thm]{Proposition}
\newtheorem{cor}[thm]{Corollary}
\newtheorem{defn}[thm]{Definition}

\newtheorem{innercustomthm}{Theorem}
\newenvironment{customthm}[1]
  {\renewcommand\theinnercustomthm{#1}\innercustomthm}
  {\endinnercustomthm}

\newcommand{\F}{\mathcal{F}}

\def\sm{{\backslash}}
\newcommand{\ZZ}{\mathbb{Z}}
\newcommand{\NN}{\mathbb{N}}
\newcommand{\eps}{\varepsilon}
\renewcommand{\c}{{\rm c}}

\title{$\ell$-covering $k$-hypergraphs are quasi-eulerian}
\author{Mateja \v{S}ajna\footnote{Corresponding author. Email: msajna@uottawa.ca. Mailing address: Department of Mathematics and Statistics, University of Ottawa, 150 Louis-Pasteur Private, Ottawa, ON, K1N 6N5, Canada.} \; and Andrew Wagner \\ {\small University of Ottawa}}

\begin{document}

\maketitle
\begin{abstract}
An {\em Euler tour} in a hypergraph $H$ is a closed walk that traverses each edge of $H$ exactly once, and an {\em Euler family} is a family of closed walks that jointly traverse each edge of $H$ exactly once. An {\em $\ell$-covering $k$-hypergraph}, for $2 \le \ell < k$, is a $k$-uniform hypergraph in which every $\ell$-subset of vertices lie together in at least one edge.

In this paper we prove that every $\ell$-covering $k$-hypergraph, for $k \ge 3$, admits an Euler family.

\end{abstract}


\baselineskip 18pt

\section{Introduction}

The complete characterization of graphs that admit an Euler tour is a classic result covered by any introductory graph theory course. The concept naturally extends to hypergraphs; that is, an Euler tour of a hypergraph is a closed walk that traverses every edge exactly once.  However, the study of eulerian hypegraphs is a much newer and largely unexplored territory.

The first results on Euler tours in hypergraphs were obtained by Lonc and Naroski \cite{LN}. Most notably, they showed that the problem of existence of an Euler tour is NP-complete on the set of $k$-uniform hypergraphs, for any $k\geq 3$, as well as when restricted to a particular subclass of 3-uniform hypergraphs.

Bahmanian and \v{S}ajna \cite{BS2} attempted a systematic study of eulerian properties of general hypergraphs; some of their techniques and results will be used in this paper. In particular, they introduced the notion of an {\em Euler family} --- a collection of closed walks that jointly traverse each edge exactly once --- and showed that the problem of existence of an Euler family is polynomial on the class of all hypergraphs.

In this paper, we define an {\em $\ell$-covering $k$-hypergraph}, for $2 \le \ell < k$, to be a non-empty $k$-uniform hypergraph in which every $\ell$-subset of vertices appear together in at least one edge.

In  \cite{BS2}, the authors proved that every 2-covering 3-hypergraph with at least two edges admits an Euler family, and the present authors gave a short proof \cite{SW} to show that every triple system --- that is, a 3-uniform hypergraph in which every pair of vertices lie together in the same number of edges --- admits an Euler tour as long as it has at least two edges. Most recently, the present authors proved the following result.

\medskip

\begin{thm}\label{thm:coveringinduction}\cite{SW1}
Let $k \ge 3$, and let $H$ be a $(k-1)$-covering $k$-hypergraph.  Then $H$ admits an Euler tour if and only if it has at least two edges.
\end{thm}

In this paper, we aim to extend Theorem~\ref{thm:coveringinduction} to all $\ell$-covering $k$-hypergraphs. Our main result is as follows.

\medskip

\begin{thm}\label{thm:lcoveringkhypergraphs}
Let $\ell$ and $k$ be integers, $2 \le \ell < k$, and let $H$ be an $\ell$-covering $k$-hypergraph.  Then $H$ admits an Euler family if and only if it has at least two edges.
\end{thm}

As the concept of an Euler family is a relaxation of the concept of an Euler tour, the conclusion of Theorem~\ref{thm:lcoveringkhypergraphs} is weaker than that of Theorem~\ref{thm:coveringinduction}; however, it holds for a much larger class of hypergraphs.

We prove Theorem~\ref{thm:lcoveringkhypergraphs} by induction on $\ell$. The base case $\ell=2$ is stated as Theorem~\ref{thm:2coveringkhypergraphs}; its proof is essentially a counting argument and requires  most of the work.
The main part of the proof is presented in Section~\ref{sec:main}, while some special cases and technical details are handled in Sections~\ref{sec:lemmas} and \ref{sec:Lovasz}. In particular, in Section~\ref{sec:Lovasz}, using the Lovasz $(g,f)$-factor theorem, we develop a sufficient condition for a $k$-uniform hypergraph without cut edges to admit an Euler family.

\section{Preliminaries}\label{sec:preliminaries}

We use hypergraph terminology established in \cite{BS1, BS2}, which applies to loopless graphs as well.  Any graph theory terms not explained here can be found in \cite{BM}.

A {\em hypergraph} $H$ is a pair $(V,E)$, where $V$ is a non-empty set, and $E$ is a multiset of elements from $2^V$.  The elements of $V=V(H)$ and $E=E(H)$ are called the {\em vertices} and {\em edges} of $H$, respectively. The {\em order} of $H$ is $|V|$, and the {\em size} is $|E|$.  A hypergraph of order 1 is called {\em trivial}, and a hypergraph with no edges is called {\em empty}.

Distinct vertices $u$ and $v$ in a hypergraph $H=(V,E)$ are called {\em adjacent} (or {\em neighbours})  if they lie in the same edge, while a vertex $v$ and an edge $e$ are said to be {\em incident} if $v\in e$.  The {\em degree} of $v$ in $H$, denoted $\deg_H(v)$, is the number of edges of $H$ incident with $v$. An edge $e$ is said to {\em cover} the vertex pair $\{ u, v \}$ if $\{ u, v \} \subseteq e$.  A hypergraph $H$ is called {\em $k$-uniform} if every edge of $H$ has cardinality $k$.

\medskip

\begin{defn}{\rm
Let $\ell$ and $k$ be integers, $2 \le \ell < k$. An {\em $\ell$-covering $k$-hypergraph} is a $k$-uniform hypergraph in which every $\ell$-subset of vertices lie together in at least one edge.
}
\end{defn}

The {\em incidence graph} of a hypergraph $H=(V,E)$ is a bipartite simple graph $G$ with vertex set $V \cup E$ and bipartition $\{V,E\}$ such that
vertices $v\in V$ and $e\in E$ of $G$ are adjacent if and only if $v$ is incident with $e$ in $H$.  The elements of $V$ and $E$ are called {\em v-vertices} and {\em e-vertices} of $G$, respectively.

A hypergraph $H' = (V',E')$ is called a {\em subhypergraph} of the hypergraph $H = (V,E)$ if $V'\subseteq V$ and $E'=\{ e \cap V': e \in E''\}$ for some submultiset $E''$ of $E$. For $e \in E$, the symbol $H \sm e$  denotes the subhypergraph $(V, E-\{ e \})$ of $H$, and for $v \in V$, the symbol $H-v$ denotes the subhypergraph $(V-\{ v \},E')$ where $E' = \{e - \{ v \}: e \in E, e - \{ v \} \ne \emptyset \}$.

A {\em $(v_0,v_k)$-walk} in $H$ is a sequence  $W=v_0e_1v_1e_2 \ldots e_kv_k$ such that $v_0,\ldots , v_k\in V$; $e_1, \ldots , e_k\in E$; and  $v_{i-1},v_i\in e_i$ with $v_{i-1}\neq v_i$ for all $i=1, \ldots, k$.  A walk is said to {\em traverse} each of the vertices and edges in the sequence.  The vertices $v_0,v_1, \ldots, v_k$ are called the {\em anchors} of $W$.  If $e_1, e_2, \ldots, e_k$ are pairwise distinct, then $W$ is called a {\em trail} ({\em strict trail} in \cite{BS1,BS2}); if $v_0=v_k$ and $k \ge 2$, then $W$ is {\em closed}.

A hypergraph $H$ is {\em connected} if every pair of vertices are connected in $H$; that is, if for any pair $u,v\in V(H)$, there exists a $(u,v)$-walk in $H$.  A {\em connected component} of $H$ is a maximal connected subhypergraph of $H$ without empty edges. The number of connected components of $H$ is denoted by $\c(H)$. We call $v \in V(H)$ a {\em cut vertex} of $H$, and  $e\in E(H)$ a {\em cut edge} of $H$, if $\c(H-v)>\c(H)$ and $\c(H\sm  e)>\c(H)$, respectively.

An {\em Euler family} of a hypergraph $H$ is a collection of pairwise anchor-disjoint and edge-disjoint closed trails that jointly traverse every edge of $H$, and an {\em Euler tour}  is a closed trail that traverses every edge of $H$.
A hypergraph that is either empty or admits an Euler tour (family) is called {\em eulerian (quasi-eulerian)}. Note that an Euler tour corresponds to an Euler family of cardinality 1, so every eulerian hypergraph is also quasi-eulerian.

The following theorem allows us to determine whether a hypergraph is eulerian or quasi-eulerian from its incidence graph.

\medskip

\begin{thm}\label{thm:incidencegraph}{\rm \cite[Theorem 2.18]{BS2}} Let $H$ be a hypergraph and $G$ its incidence graph.  Then the following hold.
\begin{description}
\item[(1)] $H$ is quasi-eulerian if and only if $G$ has a spanning subgraph $G'$ such that $\deg_{G'}(e) = 2$ for all $e\in E(H)$, and $\deg_{G'}(v)$ is even for all $v\in V(H)$.
\item[(2)] $H$ is eulerian if and only if $G$ has a spanning subgraph $G'$ with at most one non-trivial connected component
    such that $\deg_{G'}(e) = 2$ for all $e\in E(H)$, and $\deg_{G'}(v)$ is even for all $v\in V(H)$.
\end{description}
\end{thm}

\section{Technical Lemmas}\label{sec:lemmas}

In this section, we take care of some special cases and prove some technical results that will aid in the proof of our base case, Theorem~\ref{thm:2coveringkhypergraphs}.

\medskip

\begin{lem}\label{lem:2-connected} Let $k\geq 4$, and let $H$ be a 2-covering $k$-hypergraph with at least 2 edges.  Then $H$ has no cut edges.
\end{lem}

\begin{proof}
Suppose $e$ is a cut edge of $H$. Then there exist vertices $u,v \in e$ that are disconnected in $H \sm e$. Since $H$ has at least 2 edges, it must be that $k \ne |V(H)|$ and $e \ne V(H)$. Hence there exists $w \in V(H)-e$. Let $e_1, e_2$ be edges of $H$ containing $u$ and $w$, and $v$ and $w$, respectively. As $e \not\in \{ e_1,e_2 \}$, we can see that $u e_1 w e_2 v$ is a $(u,v)$-walk in $H \sm e$, a contradiction.
\end{proof}

\smallskip

\begin{lem}\label{lem:numedges} Let $k\geq 4$, and let $H$ be a 2-covering $k$-hypergraph of order $n>\frac{3k}{2}$ and size $m \ge 2$.
Then $m \ge 2\lfloor\frac{n+3}{k}\rfloor$.
\end{lem}

\begin{proof}
If  $n\leq 2k-4$, then $2\lfloor\frac{n+3}{k}\rfloor\leq 2 \le m$. Hence assume $n \ge 2k-3$.

Suppose first that  $n\geq 3k-3$.
Since there are $\binom{n}{2}$ pairs of vertices to cover, and each edge covers $\binom{k}{2}$ pairs, we know that $m\geq \frac{n(n-1)}{k(k-1)}$. As $k \ge 4$, we have
\begin{align*}
m\geq \frac{n(n-1)}{k(k-1)} &\geq \frac{(3k-3)(n-1)}{k(k-1)}= \frac{3(n-1)}{k}= \frac{2n + n - 3}{k}\\
&\geq \frac{2n + 3k - 6}{k}\geq \frac{2n + 6}{k}\geq 2\Big\lfloor\frac{n+3}{k}\Big\rfloor.
\end{align*}

Finally, assume $2k-3\leq n\leq 3k-4$.  As $2\lfloor\frac{n+3}{k}\rfloor\leq 4$, it suffices to show that $m\geq 4$.  Suppose $m\leq 3$.  Since $H$ is a 2-covering $k$-hypergraph with $n > k$ and $m \ge 2$, every vertex has degree at least 2. Thus
$$2n \leq \sum_{v\in V(H)} \deg(v) = km \le 3k$$
and $n \le \frac{3k}{2}$, contradicting the assumption that $n>\frac{3k}{2}$.

Therefore, in all cases we have $m \ge 2\lfloor\frac{n+3}{k}\rfloor.$
\end{proof}

\smallskip

\begin{lem}\label{lem:2-intersection} Let $H$ be a hypergraph with $|E(H)|\geq 2$ satifying the following.
\begin{itemize}
\item For all $e,f\in E(H)$, we have $|e\cap f|\geq 2$; and
\item there exist distinct $e, f\in E(H)$ such that $|e\cap f|\geq 3$.
\end{itemize}
Then $H$ is eulerian.
\end{lem}

\begin{proof}
Let $E(H)=\{e_1, \ldots, e_m\}$ and assume $e_1$ and $e_m$ are distinct edges such that $|e_1\cap e_m|\geq 3$.
Take any $v_1\in e_1\cap e_2$. For $i=2, \ldots, m-1$, let $v_i$ be a vertex in $(e_{i}\cap e_{i+1})- \{v_{i-1}\}$, and let $v_0 \in (e_1 \cap e_m) - \{v_1, v_{m-1}\}$. It is easy to verify that  $v_0 e_1 v_1 \ldots v_{m-1} e_m v_0$ is an Euler tour of $H$.
\end{proof}

\smallskip

\begin{cor}\label{cor:smallcases}Let $H$ be a 2-covering $k$-hypergraph of order $n$.  If $n\leq 2k-3$ or $(k,n)=(4,6)$, then $H$ is eulerian.
\end{cor}

\begin{proof}
If $n\le 2k-3$, then every pair of edges $e,f\in E(H)$ satisfies $|e\cap f|\geq 3$, so $H$ is eulerian by Lemma~\ref{lem:2-intersection}.

Assume now that $(k,n)=(4,6)$. For all $e,f \in E(H)$, we have $|e\cap f|\geq 2$.  If there exist distinct edges $e, f \in E(H)$ such that $|e\cap f|\geq 3$, then  $H$ is eulerian by Lemma~\ref{lem:2-intersection}.  Hence assume $|e \cap f|=2$ for all $e,f \in E(H)$, and let $V(H)=\{ v_1,\ldots, v_6\}$. It is not difficult to see that we must have $E(H)=\{ e_1, e_2, e_3 \}$ where, without loss of generality, the edges are $e_1=v_1v_2v_3v_4$, $e_2=v_1v_2v_5v_6$, and $e_3=v_3v_4v_5v_6$. It follows that $W=v_3e_1v_2e_2v_5e_3v_3$ is an Euler tour of $H$.
\end{proof}

\smallskip

\begin{lem}\label{lem:notturan} Let $n,k,q \in \ZZ^+$ be such that $n \geq qk$. Let
$$S=\big\{ (x_1,\ldots, x_q) \in (\ZZ^+)^q: x_1 + \cdots + x_q = n, x_i \ge k \mbox{ for all } i \big\},$$
and define $f: S \to \ZZ^+$  by $f(x_1, \ldots, x_q) = \binom{x_1}{2} + \cdots + \binom{x_q}{2}$.
Then $f$ attains its maximum on $S$ at the point $\big(k,\ldots,k,n-k(q-1) \big)$.
\end{lem}

\begin{proof}
Since the domain $S$ is finite, function $f$ indeed attains a maximum on $S$.

Let ${\bf x} =(x_1, \ldots, x_q) \in S$ be such that $f({\bf x})$ is maximum. By symmetry of $f$, we may assume that  $x_1 \leq  x_2 \leq \ldots \leq x_q$. As $x_1 \ge k$ and $x_q=n-(x_1+\ldots +x_{q-1})$, we observe that $x_q \le n-k(q-1)$.

Suppose that $x_q < n-k(q-1)$. Then there exists $i \in \{ 1,\ldots,q-1 \}$ such that $x_i>k$.  Let $i$ be the smallest index with this property, and let
$${\bf y}=(x_1, \ldots, x_{i-1}, x_i - 1, x_{i+1}, \ldots, x_{q-1}, x_q+1).$$
Then ${\bf y} \in S$ and
\begin{align*}
f({\bf y}) =& \sum_{\substack{j=1 \\ j \neq i}}^{q-1} \binom{x_j}{2} + \binom{x_i - 1}{2} + \binom{x_q + 1}{2} \\
=& \sum_{\substack{j=1 \\ j\neq i}}^{q-1} \binom{x_j}{2} + \frac{x_i(x_i - 1)}{2} - \frac{2(x_i-1)}{2} + \frac{x_q(x_q - 1)}{2} + \frac{2x_q}{2}\\
=& \sum_{j=1}^q \binom{x_j}{2} + (x_q-x_i+1)
> f({\bf x}),
\end{align*}
contradicting the choice of ${\bf x}$.

Hence $x_q = n-k(q-1)$, and consequently $x_1= \ldots = x_{q-1} = k$. Thus $f$ attains its maximum on $S$ at the point ${\bf x}=\big(k,\ldots,k,n-k(q-1) \big)$ as claimed.
\end{proof}

\section{A sufficient condition}\label{sec:Lovasz}

In this section, we state and prove Proposition~\ref{lem:relax}, which gives a sufficient condition for a $k$-uniform hypergraph to admit an Euler family. This sufficient condition will be our main tool in the proof of Theorem~\ref{thm:2coveringkhypergraphs}. It is  based on the $(g,f)$-factor theorem by Lov\'{a}sz \cite{L2}, stated below as Theorem~\ref{thm:lovasz}.

For a graph $G$ and functions $f, g : V (G) \to \NN$, a {\em $(g,f)$-factor} of $G$ is a spanning subgraph $F$ of $G$ such that $g(x) \le \deg_F(x) \le f(x)$ for all $x \in V (G)$. An {\em $f$-factor} is simply an $(f,f)$-factor. For any sets $U,W \subseteq V(G)$, let $\eps_{G}(U,W)$ denote the number of edges of $G$ with one endpoint in $U$ and the other in $W$.

\medskip

\begin{thm}\cite{L2} \label{thm:lovasz}
Let $G=(V,E)$ be a graph and $f,g: V \rightarrow \NN$ be functions such that $g(x)\leq f(x)$ and $g(x)\equiv f(x)$ (mod 2) for all $x\in V$.  Then $G$ has a $(g,f)$-factor $F$ such that $\deg_F(x)\equiv f(x)$ (mod 2) for all $x\in V$ if and only if, for all disjoint $S,T\subseteq V$, we have
\begin{equation}\label{eqn:lovaszcondition}
\sum_{x\in S}f(x) + \sum_{x\in T}(\deg_G(x) - g(x)) - \eps_G(S,T) - q(S,T)\geq 0,
\end{equation}
where  $q(S,T)$ is the number of connected components $C$ of $G-(S\cup T)$ such that $$\sum_{x\in V(C)}f(x) + \eps_G(V(C),T)\equiv 1\text{ (mod 2).}$$
\end{thm}

\medskip

\begin{prop}\label{lem:relax}
Let $k\geq 3$, and let $H=(V,E)$ be a $k$-uniform hypergraph of order $n$ and size $m$. Let $G$ be the incidence graph of $H$, and $G^*$ the graph obtained from $G$ by appending $2(m+n)^2$ loops to every v-vertex.

Assume that $H$ has no cut edges and that for all $X \subseteq E$ with $|X|\geq 2$, we have that $|X|\geq 2\lfloor\frac{\c(G^*-X)+3}{k}\rfloor$.  Then $H$ is quasi-eulerian.
\end{prop}

\begin{proof}
Let $r=2(m+n)^2$, and define $f:V(G^*)\rightarrow \ZZ$ by
$$ f(x) = \left\{
      	\begin{array}{l l}
      		r & \mbox{ if } x \in V, \\
            2 &  \mbox{ if } x \in E.
			\end{array}
		\right.
$$
We shall use Theorem~\ref{thm:lovasz} to show that $G^*$ has an $(f,f)$-factor, so let $S,T \subseteq V(G^*)$ be disjoint sets, and denote
$$\gamma(S,T)= \sum_{x\in S} f(x) + \sum_{x\in T}(\deg_{G^*}(x)-f(x)) - \eps_{G^*}(S,T)-q(S,T),$$
where $q(S,T)$ is the number of connected components $C$ of $G^*-(S\cup T)$ such that $\eps_{G^*}(V(C), T)$ is odd.
Observe that Condition~(\ref{eqn:lovaszcondition}) for $G^*$ with $g=f$ is equivalent to $\gamma(S,T) \ge 0$.

Since $G$ is a subgraph of $K_{n,m}$, we have $\eps_{G^*}(S,T)\le mn$ and $q(S,T)\le m+n$, and therefore $\eps_{G^*}(S,T) + q(S,T) \le (m+n)^2=\frac{r}{2}$. In addition, we have $\deg_{G^*}(x) - f(x) \geq r$ for all $x\in V$, and $\deg_{G^*}(x) - f(x) \geq k-2$ for all $x\in E$.

{\em Case 1: $(S \cup T) \cap V \ne \emptyset$.}
If $S \cap V \ne \emptyset$, then
$\displaystyle \sum_{x\in S} f(x) \ge r$, and
if $T \cap V \ne \emptyset$, then
$\displaystyle \sum_{x\in T} (\deg_{G^*}(x)-f(x)) \ge r$.
Thus, in both cases
\begin{align*}
\gamma(S,T) &= \Big( \sum_{x\in S} f(x) + \sum_{x\in T} (\deg_{G^*}(x)-f(x))  \Big) - \Big( \eps_{G^*}(S,T) + q(S,T) \Big)\\
& \geq r -\frac{r}{2}  \ge 0.
\end{align*}

{\em Case 2: $(S \cup T) \cap V =\emptyset$.}  Then $\eps_{G^*}(S,T) = 0$ since $S \cup T \subseteq E$.

First, suppose $T=\emptyset$. Then $\eps_{G^*}(V(C),T) = 0$ for all connected components $C$ of $G^*-(S\cup T)$, so $q(S, T) = 0$.  Hence
$\displaystyle \gamma(S,T) = \sum_{x\in S} f(x) \geq 0$.

Next, suppose $S=\emptyset$ and $|T| = 1$.  Then $S \cup T=\{ e \}$ for some $e \in E$. By assumption, edge $e$ is not a cut edge of $H$ and hence by \cite[Theorem 3.23]{BS1}, e-vertex $e$ is not a cut vertex of $G^*$, and $G^*-(S \cup T)$ is connected. It follows that $q(S,T) \le 1$ and
$$
\gamma(S,T) = (\deg_{G^*}(e)-f(e))  - q(S,T) \ge (k-2)  - 1  \ge 0.
$$

We may now assume that $T \ne \emptyset$ and $|S \cup T|\ge 2$.
Since each connected component $C$ of $G^*-(S \cup T)$ with $\eps_{G^*}(V(C),T)$ odd corresponds to at least one edge incident with a vertex in $T$, the number of such components is at most $k|T|$. Hence
$q(S,T)\le \min \{ \c(G^*-(S \cup T)),k|T| \}$, and
\begin{align}
\gamma(S,T) &=2|S| + (k-2)|T| - q(S,T) \nonumber\\
&\geq 2|S\cup T| + (k-4)|T| - \min \{ \c(G^*-(S \cup T)),k|T| \}. \label{eq:gamma}
\end{align}

Define $t = \lfloor\frac{\c(G^*-(S\cup T)) + 3}{k}\rfloor$, so that
$$kt-3\leq \c(G^*-(S\cup T))\leq kt+k-4.$$
If $|T| \ge t+1$, then $\min \{ \c(G^*-(S \cup T)),k|T| \} = \c(G^*-(S \cup T)) \le kt+k-4$, so Inequality~(\ref{eq:gamma}) yields
$$
\gamma(S,T) \ge 2|S\cup T| + (k-4)(t + 1) - (kt+k-4) = 2|S\cup T| - 4t.
$$
The same bound is obtained if $|T| \le t$: in this case, we have
$\min \{ \c(G^*-(S \cup T)),k|T| \} \le k|T|$, so that (\ref{eq:gamma}) yields
$$
\gamma(S,T) \ge 2|S\cup T| + (k-4)|T| - k|T| = 2|S\cup T| - 4|T| \ge 2|S\cup T| - 4t.
$$
In both cases, as $S\cup T \subseteq E$ and $|S\cup T| \ge 2$, the assumption of the proposition implies $|S\cup T| \ge 2\lfloor\frac{\c(G^*-(S\cup T)) + 3}{k}\rfloor=2t$, so  that $\gamma(S,T)\geq 0$.

Therefore, $\gamma(S,T)\geq 0$ for all disjoint $S,T \subseteq V(G^*)$, and by Theorem~\ref{thm:lovasz}, we conclude that  $G^*$ has an $(f,f)$-factor $F$. Deleting the loops of $F$, we obtain a spanning subgraph $F'$ of $G$ in which all v-vertices have even degree and all e-vertices have degree 2.  Thus $H$ admits an Euler family  by Theorem~\ref{thm:incidencegraph}.
\end{proof}

\section{Proof of the main result}\label{sec:main}

We shall now prove our main result, Theorem~\ref{thm:lcoveringkhypergraphs}. We use induction on $\ell$, and most of the work is required to prove the basis of induction, which we state below as Theorem~\ref{thm:2coveringkhypergraphs}.

\medskip

\begin{thm}\label{thm:2coveringkhypergraphs}
Let $k\geq 4$, and let $H$ be a 2-covering $k$-hypergraph with at least two edges.  Then $H$ is quasi-eulerian.
\end{thm}

\begin{proof}
Let $H=(V,E)$ with $n=|V|$ and $m=|E|$. If $n\leq 2k-3$, then $H$ is eulerian by Corollary~\ref{cor:smallcases}, so we may assume that $n\geq 2k-2.$

If $n\leq\frac{3k}{2}$, it then follows that $(k,n)=(4,6)$.  Again, $H$ is eulerian by Corollary~\ref{cor:smallcases}. Hence  $n>\frac{3k}{2}$, and  Lemma~\ref{lem:numedges} implies that
$m \geq 2\big\lfloor\frac{n+3}{k}\big\rfloor$.

In the rest of the proof we show that $H$ satisfies the conditions of Proposition~\ref{lem:relax}.

Let $G^*$ be the graph obtained from  the incidence graph of $H$ by adjoining $2(m+n)^2$ loops to every v-vertex.

Fix any $X \subseteq E$ with $|X|\geq 2$, and denote $q=\c(G^*-X)$.

Suppose that $|X| < 2\Big\lfloor\frac{q+3}{k}\Big\rfloor$.
If $q \leq 2k-4$, then this supposition  implies that $|X| < 2$,  a contradiction. Hence we may assume that $q\geq 2k-3$, and hence $q \ge 5$. Moreover, our supposition implies
\begin{equation}\label{eq:X}
|X| \le 2\frac{q+3}{k}-1.
\end{equation}


Let $\ell$ denote the number of v-vertices that are isolated in $G^*-X$.

{\em Case 1: $\ell\geq 1$.}  If $\ell=n$, then $X=E$, $q = n$, and $|X|=|E| \geq 2\lfloor \frac{n+3}{k} \rfloor = 2\lfloor\frac{q + 3}{2}\rfloor$, contradicting our assumption on $X$.  Thus we may assume $\ell<n$, and hence $\ell <q$.

Since $G^*-X$ has  $q-\ell$ non-trivial connected components, each with at least $k$ v-vertices, we have
\begin{equation}\label{eqn:n0}
 n \geq \ell + k(q-\ell).
\end{equation}
Since $q>\ell$, this inequality also implies
\begin{equation}\label{eqn:n2}
 n\geq\ell + k.
\end{equation}

Let $S$ be the set of pairs $\{ u,v \}$ of v-vertices such that $u$ is isolated in $G^*-X$, and $v$ is not. Then $|S|=\ell(n-\ell)$. Observe that every edge of $H$ covers at most $\frac{k^2}{4}$ pairs from $S$, which implies that
$|X| \geq \frac{\ell(n-\ell)}{\frac{k^2}{4}}$.
Combining this inequality with (\ref{eq:X}), we obtain
\begin{equation}\label{eqn:master}
\frac{4\ell(n-\ell)}{k^2} \leq \frac{2q+6-k}{k}.
\end{equation}
Substituting $q \leq \ell+\frac{n-\ell}{k}$ from Inequality (\ref{eqn:n0}) and rearranging yields
$$ n(4\ell -2) \leq 4\ell^2 - k^2 + 2\ell k - 2\ell + 6k.$$
Further substituting $n\geq\ell + k$ from~(\ref{eqn:n2}) and isolating $\ell$, we obtain $\ell \leq 4 - \frac{k}{2}$, which implies
$\ell \in \{ 1,2 \}$ as $k \ge 4$.

However, if on the left-hand side of Inequality (\ref{eqn:master}) we apply $\frac{n-\ell}{k} \geq q - \ell$ from~(\ref{eqn:n0}) and simplify, then we obtain
$$(4\ell - 2)q - 4\ell^2 \leq 6 - k \le 2.$$
Now substituting either $\ell=1$ or $\ell=2$ yields $q\leq 3$, a contradiction.

{\em Case 2: $\ell=0$.}  Let $C_1,C_2,\dotso , C_q$ be the connected components of $G^*-X$, and let $n_i$ denote the number of v-vertices of $C_i$. Note that $n_i \ge k$ for all $i$.

The number of pairs of v-vertices  that lie in distinct connected components of $G^* -X$ is
$ \binom{n}{2} - \sum_{i=1}^{q}\binom{n_i}{2}$, and
these pairs must all be covered by the edges of $X$. As $n \ge qk$, $n_1+\ldots+n_q=n$, and $n_i \ge k$, for all $i$,we know that
$\sum_{i=1}^{q}\binom{n_i}{2} \le (q-1)\binom{k}{2} + \binom{n-k(q-1)}{2}$
by Lemma~\ref{lem:notturan}.
Therefore,
$$
\binom{n}{2} - \sum_{i=1}^{q}\binom{n_i}{2} \geq \binom{n}{2} - (q-1)\binom{k}{2}-\binom{n-k(q-1)}{2}.
$$
Since each edge of $X$ covers up to $\binom{k}{2}$ pairs of v-vertices in distinct connected components, we deduce that
$$
|X| \ge \frac{\binom{n}{2} - (q-1)\binom{k}{2}-\binom{n-k(q-1)}{2}}{\binom{k}{2}}.$$

On the other hand, by (\ref{eq:X}),  we have $|X|\leq\frac{2q+6-k}{k}$, so \begin{equation}\label{eq:case2}
\frac{\binom{n}{2} - (q-1)\binom{k}{2}-\binom{n-k(q-1)}{2}}{\binom{k}{2}}\leq\frac{2q+6-k}{k}.
\end{equation}

We now substitute $x=q-1$, noting that $x \ge 4$ as $q \ge 5$. Rearranging Inequality (\ref{eq:case2}), we then obtain
$$2kxn \le k^2x^2 + (k^2+2k-2)x-(k-8)(k-1).$$
Applying $n \ge qk=(x+1)k$ further yields
$$k^2x^2 +(k^2-2k+2)x+(k-8)(k-1) \le 0.$$
Denote the left-hand side by $f(x)=ax^2+bx+c$, where $a=k^2$, $b=k^2-2k+2$, and $c=(k-8)(k-1)$, and observe that $a, b > 0$ as $k \ge 4$. If $b^2-4ac<0$, then $f(x)>0$ for all $x$, a contradiction. Hence assume $b^2-4ac \ge 0$. Let $x_2$ be the larger of the two roots of $f(x)=0$. If $x_2<4$, then $f(x)>0$ for all $x \ge 4$, a contradiction. Hence we must have
$$4 \le \frac{-b+\sqrt{b^2-4ac}}{2a}.$$
Since $a,b >0$, it is straightforward to show that
$16a+4b+c \le 0$ follows. However,
$$ 16a+4b+c = k(21k-17)+16 >0,$$
a contradiction.

Since each case leads to a contradiction, we conclude that $|X|\geq \lfloor\frac{\c(G^*-X)+3}{k}\rfloor$. By Lemma~\ref{lem:2-connected}, hypergraph $H$ has no cut edges, so we may apply Proposition~\ref{lem:relax} to conclude that $H$ is quasi-eulerian.
\end{proof}

We are now ready to prove our main result, restated below. \\

\begin{customthm}{\ref{thm:lcoveringkhypergraphs}}
Let $\ell$ and $k$ be integers, $2 \le \ell < k$, and let $H$ be an $\ell$-covering $k$-hypergraph.  Then $H$ is quasi-eulerian if and only if it has at least two edges.
\end{customthm}

\begin{proof}
Since $H$ is non-empty, and since a hypergraph with a single edge does not admit a closed trail, necessity is easy to see.

To prove sufficiency, for $s \geq 1$ and $\ell \geq 2$, define the proposition
$$P_{s}(\ell): \mbox{ ``Every } \ell \mbox{-covering } (\ell+s) \mbox{-hypergraph with at least two edges is quasi-eulerian.''}$$
Theorem~\ref{thm:coveringinduction} implies that $P_{1}(\ell)$ holds for all $\ell \geq 2$.  Hence fix any $s \geq 2$.

We prove $P_{s}(\ell)$ by induction on $\ell$. As $\ell+s \ge 4$, the basis of induction, $P_{s}(2)$, follows from Theorem~\ref{thm:2coveringkhypergraphs}.  Suppose that, for some $\ell \geq 2$, the proposition $P_{s}(\ell)$ holds; that is, every $\ell$-covering $(\ell+s)$-hypergraph with at least two edges is quasi-eulerian.

Let $H=(V,E)$ be an $(\ell+1)$-covering $\big( (\ell+1 \big) +s)$-hypergraph with $|E| \geq 2.$  Fix any $v\in V$ and let $V^*=V-\{ v \}$. Define a mapping $\varphi: E \to 2^{V^*}$ by
$$\varphi(e)=e-\{ v \} \quad \mbox{ if } v \in e,$$
and otherwise,
$$\varphi(e)=e-\{ u \} \quad \mbox{ for any } u \in e.$$
Then let $E^*=\{ \varphi(e): e \in E \}$ and $H^*=(V^*,E^*)$, so that $\varphi$ is a bijection from $E$ to $E^*$. It is straightforward to verify that $H^*$ is an $\ell$-covering $(\ell+s)$-hypergraph. As $|E^*|=|E| \ge 2$, by induction hypothesis, hypergraph $H^*$ admits an Euler family $\F^*$. In each closed trail in $\F^*$, replace each $e \in E^*$ with $\varphi^{-1}(e)$ to obtain a set $\F$ of closed trails of $H$. It is not difficult to verify that $\F$ is an Euler family of $H$, so $P_s(\ell+1)$ follows.

By induction, we conclude that $P_{s}(\ell)$ holds for all $\ell \geq 2$, and any $s \ge 1$. Therefore, every $\ell$-covering $k$-hypergraph with at least two edges is quasi-eulerian.
\end{proof}

\small

\section*{Acknowledgements}

The first author gratefully acknowledges support by the Natural Sciences and Engineering Research Council of Canada (NSERC), Discovery Grant RGPIN-2016-04798.

\end{document}